\documentclass[11pt,a4paper, xcolor]{article}

\usepackage{url,lineno}
\usepackage{amsfonts}
\usepackage{amssymb,amsthm,amsmath,graphicx,bm,ebezier,mathtools}
\usepackage{hyperref}
\usepackage{color,enumerate}
\usepackage{url,lineno}
\usepackage{arydshln}
\usepackage{lscape}
\usepackage[ruled,vlined]{algorithm2e}
\usepackage{soul} 
\usepackage{tikz}

\newtheorem{theorem}{Theorem}
\newtheorem{definition}[theorem]{Definition}
\newtheorem{proposition}[theorem]{Proposition}
\newtheorem{corollary}[theorem]{Corollary}
\newtheorem{lemma}[theorem]{Lemma}

\theoremstyle{remark}

\newtheorem{example}[theorem]{Example}

\def\e{\mathrm{e}}

\def\CaA{\mathcal{A}}
\def\CaD{\mathcal{D}}

\def\CaH{\mathcal{H}}
\def\CaC{\mathcal{C}}

\def\CaX{\mathcal{X}}

\def\N{\mathbb{N}}
\def\Z{\mathbb{Z}}
\def\R{\mathbb{R}}
\def\Q{\mathbb{Q}}

\title{Characterizing affine $\mathcal C$-semigroups}
\date{}
\author{
J. D. D\'{\i}az-Ram\'{\i}rez
\footnote{
Departamento de Matem\'aticas/INDESS (Instituto Universitario para el Desarrollo Social Sostenible), Universidad de C\'adiz,
E-11406 Jerez de la Frontera (C\'{a}diz, Spain).
E-mail: juandios.diaz@uca.es.}\\
J. I. Garc\'{\i}a-Garc\'{\i}a
\footnote{
Departamento de Matem\'aticas/INDESS (Instituto Universitario para el Desarrollo Social Sostenible),
Universidad de C\'adiz, E-11510 Puerto Real  (C\'{a}diz, Spain).
E-mail: ignacio.garcia@uca.es.}\\
D. Mar\'{\i}n-Arag\'on
 \footnote{
     Departamento de Matem\'aticas,
     Universidad de C\'adiz, E-11510 Puerto Real  (C\'{a}diz, Spain).
     E-mail: daniel.marin@uca.es.}
\\
A. Vigneron-Tenorio
\footnote{
Departamento de Matem\'aticas/INDESS (Instituto Universitario para el Desarrollo Social Sostenible), Universidad de C\'adiz,
E-11406 Jerez de la Frontera (C\'{a}diz, Spain).
E-mail: alberto.vigneron@uca.es.}
}

\begin{document}
\maketitle
\begin{abstract}
Let $\mathcal C \subset \N^p$ be a finitely generated integer cone and $S\subset \mathcal C$ be an affine semigroup such that the real cones generated by $\mathcal C$ and by $S$ are equal. The semigroup $S$ is called $\mathcal C$-semigroup if $\mathcal C\setminus S$ is a finite set.
In this paper, we characterize the $\mathcal C$-semigroups from their minimal generating sets, and we give an algorithm to check if $S$ is a $\mathcal C$-semigroup and to compute its set of gaps. We also study the embedding dimension of $\mathcal C$-semigroups obtaining a lower bound for it, and introduce some families of $\mathcal C$-semigroups whose embedding dimension reaches our bound. In the last section, we present a method to obtain a decomposition of a $\mathcal C$-semigroup into irreducible $\mathcal C$-semigroups.
\end{abstract}



\section*{Introduction}

An affine semigroup $S\subset \N^p$ is called $\CaC_S$-semigroup if $\CaC_S\setminus S$ is a finite set where $\CaC_S\subset \N^p$ is the minimal integer cone containing it. These semigroups are a
natural generalization of numerical semigroups, and several of their invariants can be generalized. For a given numerical semigroup $G$, it is well-known that $\N\setminus G$ is finite; in fact, $G\subset \N$ is a numerical semigroup if it is a submonoid of $\N$ and $\N\setminus G$ is finite (for topics related with numerical semigroups see \cite{libro_rosales} and the references therein). In general, it does not happen for affine semigroups.

$\CaC$-semigroups are introduced in \cite{Csemigroup}, where the authors study several properties about them (for example, an extended Wilf's conjecture for $\CaC$-semigroups is given). These semigroups appear in different contexts: when the integer points in an infinite family of some homothetic convex bodies in $\R^p_{\ge}$ are considered (see, for instance, \cite{politope_semig}, \cite{convex_body} and the references therein), or when the non-negative integer solutions of some modular Diophantine inequality are studied (see \cite{prop_mod}), et cetera. In case the cone $\CaC$ is $\N^p$, $\N^p$-semigroups are called generalized numerical semigroups and they were introduced in \cite{GenSemNp}. Recently, in \cite{resolucion_maxima} it is proved that the minimal free resolution of the associated algebra to any $\CaC$-semigroup has maximal projective dimension possible.

In this context, $\N^p$-semigroups are characterized in \cite{CFU}, but the general problem was opened, {\em given any affine semigroup $S$, how to detect if $S$ is or not a $\CaC_S$-semigroup?} This work's primary goal is to determine the conditions that any affine semigroup given by its minimal set of generators has to verify to be a $\CaC_S$-semigroup. We solve this problem in Theorem \ref{main_theorem}, and in Algorithm \ref{algoritmo_check_Csemig} we provide a computational way to check it.

Other open problem is to compute the set of gaps of any $\CaC$-semigroup defined by its minimal generating set. We solve this problem by means of setting a finite subset of $\CaC$ containing all the gaps of a given $\CaC$-semigroup. Algorithm \ref{computing_gaps_Csemig} computes the set of gaps of $\CaC$-semigroups.

In this paper, we also go in-depth to study the embedding dimension of $\CaC$-semigroups. In \cite[Theorem 11]{Csemigroup}, a lower bound of the embedding dimension of $\N^p$-semigroups is provided, and some families of $\N^p$-semigroups reaching this bound are given. Besides, in \cite[Conjecture 12]{Csemigroup}, it is proposed a conjecture about a lower bound for the embedding dimension of any $\CaC$-semigroup. In section \ref{sec_embedding_dimension}, we introduce a lower bound of the embedding dimension of any $\CaC$-semigroup, and some families of $\CaC$-semigroups whose embedding dimension is equal to this new bound.

An important problem in Semigroup Theory is to determine some decomposition of a semigroup into irreducible semigroups (for example, see \cite[Chapter 3]{libro_rosales} for numerical semigroups, or its generalization for $\N^p$-semigroups in \cite{irreducible_GNS}). We propose an algorithm to compute a decomposition of any $\CaC$-semigroups into irreducible $\CaC$-semigroups.

The results of this work are illustrated with several examples. To this aim,
we have used third-party software, such as Normaliz \cite{normaliz},
and the libraries {\tt CharacterizingAffineCSemigroup} and {\tt Irreducible} \cite{PROGRAMA} developed by the authors in Python \cite{python}.

The content of this work is organized as follows. Section \ref{sec_prelimiraries} introduces the initial definitions and notations used throughout the paper, mainly related to finitely generated cones. In Section \ref{sec_main}, a characterization of $\CaC$-semigroups is provided, and an algorithm to check if an affine semigroup is a $\CaC$-semigroup. Section \ref{sec_gaps} is devoted to give an algorithm to compute the set of gaps of a $\CaC$-semigroup. Section \ref{sec_embedding_dimension} makes a study of the minimal generating sets of $\CaC$-semigroups formulating explicitly a lower bound for their embedding dimensions.
Finally, in Section \ref{sec_irreducible} an algorithm for computing a decomposition of a $\CaC$-semigroup into irreducible $\CaC$-semigroups is presented.

\section{Preliminaries}\label{sec_prelimiraries}

The sets of real numbers, rational numbers, integer numbers and the non-negative integer numbers are denoted by $\R$, $\Q$, $\Z$ and $\N$, respectively. Given $A$ a subset of $\R$, $A_\ge$ is the set of elements in $A$ greater than or equal to zero. For any $n\in \N$, $[n]$ denotes the set $\{1,\ldots n\}$. Given an element $x$ in $\R^n$, $||x||_1$ denotes the sum of the absolute value of its entries, that is, its 1-norm. In this paper we assume the set $\{\mathbf{e}_1,\ldots , \mathbf{e}_p\}$ is the canonical basis of $\R^p$.

For a non empty subset of $\R_\ge ^p$, $B$, we define the cone
$$L(B):= \left\{\sum_{i=1}^n \lambda_i \mathbf{b}_i \mid n\in \N, \{\mathbf{b}_1,\ldots ,\mathbf{b}_n\}\subset B,\text{ and } \lambda_i \in \R_{\geq}, \forall i \in [n] \right\}.$$

Given a real cone $\CaC\subset \R^p_\ge$, it is well-known that $\CaC\cap \N^p$ is finitely generated if and only if there exists a rational point in each extremal ray of $\CaC$. Moreover, any subsemigroup of $\CaC$ is finitely generated if and only if there exists an element in the semigroup in each extremal ray of $\CaC$. A good monograph about rational cones and affine monoids is \cite{Bruns}.
From now on, we assume that the integer cones considered in this work are finitely generated.

\begin{definition}
Given an integer cone $\CaC\subset \N^p$, an affine semigroup $S\subset \CaC$ is said to be a $\CaC$-semigroup if $\CaC \setminus S$ is a finite set. If the cone $\CaC=\N^p$, a $\CaC$-semigroup is called $\N^p$-semigroup.
\end{definition}

Fix a finitely generated semigroup $S\subset \N^p$, we denote by $\CaC_S$ the integer cone $L(S)\cap\N^p$. Note that, if $S$ is a $\CaC$-semigroup, the cone $\CaC$ is $\CaC_S$. Obviously, a unique cone corresponds to infinite different semigroups.

The cone $L(S)$ is a polyhedron and we denote by $\{h_1(x)=0,\ldots, h_t(x)=0\}$ the set of its supported hyperplanes. We suppose $L(S)=\{ x\in \R_\ge^d \mid h_1(x)\ge 0, \ldots , h_t(x)\ge 0\}$. Unless otherwise stated, the considered coefficients of each $h_i(x)$ are integers and relatively primes.

Assume $L(S)$ has $q$ extremal rays denoted by $\tau_1,\ldots ,\tau_q$.
Then, each $\tau_i$ is determined by the set of linear equations
$H_i:=\{h_{j_1}(x)=0,\ldots, h_{j_{p-1}}(x)=0\}$ where $J_i:=\{j_1<\cdots <j_{p-1}\}\subset [t]$ is the index set of the supported hyperplanes containing $\tau_i$. So, for each $i\in [q]$, there exists the minimal non-negative integer vector $\mathbf{a}_i$ such that $\tau_i=\{\lambda\mathbf{a}_i\mid \lambda\in \R_\ge \}$. The set $\{\mathbf{a}_1,\ldots , \mathbf{a}_q\}$ is a generating set of $L(S)$.

Note that a necessary condition for $S$ to be a $\CaC_S$-semigroup is the set $\tau_i \cap (\CaC_S\setminus S)$ is finite for all $i\in [q]$.

From each extremal ray $\tau_i$ of $L(S)$, we define $\upsilon_{i}(\alpha)$ as the parallel line to $\tau_i$ given by the solutions of the linear equations $\bigcup _{j\in J_i} \{h_j(x)=\alpha_j\}$ where $\alpha=(\alpha_{j_1},\ldots ,\alpha_{j_{p-1}})\in \Z^{p-1}$. For every integer point $P\in \Z^p$ and $i\in [q]$, there exists $\alpha \in \Z^{p-1}$ such that $P$ belongs to $\upsilon_{i}(\alpha)$; if $P\in \CaC_S$, $\alpha \in \N^{p-1}$. We denotes by $\Upsilon _i(P)$ the element $(h_{j_1}(P),\ldots, h_{j_{p-1}}(P))\in\N^{p-1}$ with $J_i=\{j_1<\cdots <j_{p-1}\}$, $P\in\CaC_S$ and $i\in [q]$. Note that for any $P\in\CaC_S$, $P\in \upsilon_i(\alpha)$ if and only if $\alpha =\Upsilon_i(P)$.

Since all the semigroups appearing in this work are finitely generated, from now on, we omit the term \textit{affine} when affine semigroups are considered.

\section{An algorithm to detect if a semigroup is a $\CaC$-semigroup}\label{sec_main}

In this section, we study the conditions that a semigroup has to satisfy to be a $\CaC$-semigroup. This characterization depends on the minimal set of generators of the given semigroup.

Let $S\subset \N^p$ be the affine semigroup minimally generated by $\Lambda_S=\{\mathbf{s}_1,\ldots, \mathbf{s}_q,\mathbf{s}_{q+1}, \ldots, \mathbf{s}_n\}$ and $\tau_1,\ldots ,\tau_q$ be the extremal rays of $L(S)$. Assume that for every $i\in [q]$, $\tau_i \cap (\CaC_S\setminus S)$ is finite and $\mathbf{s}_i$ is the minimum (respect to the natural order) element in $\Lambda_S$ belonging to $\tau_i$.
We denote by $\mathbf{f}_i$ the maximal element in $\tau_i \cap (\CaC_S\setminus S)$ respect the natural order. Recall that $\mathbf{a}_i$ is the minimal non-negative integer vector defining $\tau_i$, and let $\mathbf{c}_i\in S$ be the element $\mathbf{f}_i+\mathbf{a}_i$. In case $\tau_i \cap (\CaC_S\setminus S)=\emptyset$, we fix $\mathbf{f}_i=-\mathbf{a}_i$. The elements  $\mathbf{f}_i$ and $\mathbf{c}_i$ are a generalization on the semigroup $\tau_i \cap S$ of the concepts Frobenius number and conductor of a numerical semigroup; for numerical semigroups, the Frobenius number is the maximal natural number that is not in the semigroup, and the conductor is Frobenius number plus one (see \cite[Chapter 1]{libro_rosales}).
Hence, we call Frobenius element and conductor of the semigroup $\tau_i\cap S$ to $\mathbf{f}_i$ and $\mathbf{c}_i$, respectively.
One easy but important property of $S$ is for every $P\in S$, $P+\mathbf{c}_i+\lambda\mathbf{a}_i\in S$ for any $i\in [q]$ and $\lambda\in\N$.

Note that $\tau_i\cap \N^p$ is equal to $\{\lambda\mathbf{a}_i\mid \lambda\in \N \}$. So, there exists $S_i\subset \N$ such that $\tau_i\cap S=\{\lambda\mathbf{a}_i\mid \lambda\in S_i \}$. If we assume that $\tau_i \cap (\CaC_S\setminus S)$ is finite, it is easy to prove that $S_i$ is a numerical semigroup.

\begin{lemma}\label{isomorfismo_en_rayo}
The $\tau_i$-semigroup $\tau_i \cap S$ is isomorphic to the numerical semigroup $S_i$.
\end{lemma}

\begin{proof}
Consider the isomorphism $\varphi:\tau_i \cap S\to S_i$ with $\varphi(\mathbf{w}):=\lambda$ such that $\mathbf{w}=\lambda\mathbf{a}_i$.
\end{proof}

\begin{corollary}
Given the semigroup $\tau_i \cap S$, $\mathbf{f}_i$ is equal to $f\mathbf{a}_i$ and $\mathbf{c}_i=c\,\mathbf{a}_i$ where $f$ and $c$ are the Frobenius number and the conductor of the numerical semigroup $S_i$, respectively.
\end{corollary}

To test whether $\tau_i \cap (\CaC_S\setminus S)$ is finite, the following result can be used.

\begin{lemma}
Let $S\subset \N^p$ be a semigroup and $\tau$ an extremal ray of $L(S)$ satisfying $\tau\cap \N^p=\{\lambda\mathbf{a}\mid \lambda\in \N \}$ with $\mathbf{a}\in \N^p$. Then, $\tau \cap (\CaC_S\setminus S)$ is finite if and only if $\gcd (\{\lambda\mid \lambda\mathbf{a}\in \tau\cap \Lambda_S \})=1$.
\end{lemma}

\begin{proof}
Assume that $\tau \cap (\CaC_S\setminus S)$ is finite and suppose that $\gcd (\{\lambda\mid \lambda\mathbf{a}\in \tau\cap \Lambda_S \})=n\neq 1$. Hence, every element $\lambda\mathbf{a}$ with $\gcd(n,\lambda)=1$ does not belong to $S$, and then $\tau \cap (\CaC_S\setminus S)$ is not finite.

Conversely, by Lemma \ref{isomorfismo_en_rayo}, if $\gcd (\{\lambda\mid \lambda\mathbf{a}\in \tau\cap \Lambda_S \})=1$, $S_i$ is isomorphic to $\tau_i \cap S$. Therefore, $\tau \cap (\CaC_S\setminus S)$ is finite.
\end{proof}

To introduce the announced characterization, we need to define some subsets of $L(S)$ and prove some of their properties.
Associated to the integer cone $\CaC_S$, consider the sets $\CaA:=\{ \sum_{i\in[q]} \lambda_i\mathbf{a}_i \mid 0\le \lambda_i\le 1\}\cap \N^p$ and $\CaD:=\{ \sum_{i\in[q]} \lambda_i\mathbf{s}_i \mid 0\le \lambda_i\le 1\}\cap \N^p$.

\begin{lemma}\label{lemma_diamante}
Given $P\in \CaC_S$, there exist $Q\in \CaA$ and $\beta \in \N^q$ such that $P=Q+\sum_{i\in[q]} \beta_i\mathbf{a}_i$. Moreover, $\Upsilon_j(P)=\Upsilon_j(Q)+\sum_{i\in [q]}\beta_i \Upsilon_j(\mathbf{a_i})$ for every $j\in[q]$.
\end{lemma}

\begin{proof}
Since $P\in \CaC_S$, $P=\sum_{i\in[q]}\mu_i\mathbf{a}_i$ with $\mu_i\in \Q_\ge$. For each $\mu_i$ there exists $\lambda_i\in [0,1]$ satisfying $\mu_i=\lfloor \mu_i\rfloor +\lambda_i$. Hence, $P=Q +\sum _{i\in[q]}\lfloor \mu_i\rfloor\mathbf{a}_i$ where $Q=\sum_{i\in[q]}\lambda_i\mathbf{a}_i=P-\sum _{i\in[q]}\lfloor \mu_i\rfloor\mathbf{a}_i\in\CaA$. Trivially, $\Upsilon_j(P)$ is equal to $\Upsilon_j(Q)+\sum_{i\in [q]}\beta_i \Upsilon_j(\mathbf{a}_i)$ for every $j\in[q]$.
\end{proof}

For every $i\in [q]$, consider $ T_i\subset\N^{p-1}$ the semigroup generated by the finite set $\{\Upsilon_i(Q)\mid Q\in \CaA\}$ and $\Gamma_i$ its minimal generating set. Note that the sets $\CaA$, $ T_i$ and $\Gamma_i$ only depend on the cone $\CaC_S$, and since $\mathbf{a}_i\in \CaA$, $0\in  T_i$. The relationships between the elements in $\CaC_S$ and $S$, and the elements belonging to $T_i$ and $\Gamma_i$ are explicitly determined in the following results for each $i\in [q]$.

\begin{lemma}
Let $P$ be an element in $\CaC_S$ such that $P\in \upsilon_i(\alpha)$ for some $\alpha\in \N^{p-1}$, then $\alpha\in  T_i$.
\end{lemma}

\begin{proof}
By definition, $P\in \upsilon_i(\alpha)$ means that $\alpha =\Upsilon_i(P)$. Using Lemma \ref{lemma_diamante}, $P=Q+\sum_{j\in[q]} \beta_j\mathbf{a}_j$ with $Q,\mathbf{a}_1,\ldots ,\mathbf{a}_q\in \CaA$ and $\beta_1,\ldots ,\beta_q\in \N$. Therefore, $\Upsilon_i (P)=\Upsilon_i(Q)+\sum_{j\in [q]}\beta_j \Upsilon_i(\mathbf{a}_j)\in  T_i$.
\end{proof}

\begin{corollary}
For every $\alpha \in  T_i$, $\CaC_S\cap \upsilon_i(\alpha)\neq \emptyset$ if and only if $\CaC_S\cap \upsilon_i(\beta)\neq \emptyset$ for all $\beta \in \Gamma _i$.
\end{corollary}

\begin{proof}
Since $\Gamma _i\subset  T_i$, if for all $\alpha \in  T_i$, $\CaC_S\cap \upsilon_i(\alpha)\neq \emptyset$ then $\CaC_S\cap \upsilon_i(\beta)\neq \emptyset$ for all $\beta \in \Gamma _i$.

Assume that $\CaC_S\cap \upsilon_i(\beta)\neq \emptyset$ for all $\beta \in \Gamma _i$ and let $\alpha$ be an element in $ T_i$. Then, there exist $\beta_1,\ldots ,\beta_k\in \Gamma_i$, $\mu_1,\ldots ,\mu_k\in \N$ and $Q_1,\ldots ,Q_k\in \CaD$ such that $\alpha=\sum_{j\in [k]}\mu_j\beta_j$ and $\Upsilon_i(Q_j)=\beta _j$ for $j\in [k]$. Note that $P=\sum_{j\in [k]}\mu_jQ_j\in\CaC_S$ belongs to $\upsilon_i(\alpha)$.
\end{proof}

\begin{corollary}\label{minimal_terminos_independientes}
For every $\alpha \in  T_i$, $S\cap \upsilon_i(\alpha)\neq \emptyset$ if and only if $S\cap \upsilon_i(\beta)\neq \emptyset$ for all $\beta \in \Gamma _i$.
\end{corollary}

Note that if $P\in S\cap \upsilon_i(\alpha)$ for some $\alpha \in\N^{p-1}$ and $i\in[q]$, then $P+\mathbf{c}_i+\lambda\mathbf{a}_i\in S$ and $\Upsilon_i(P+\mathbf{c}_i+\lambda\mathbf{a}_i)=\alpha$ for all $\lambda\in \N$.

Now, we introduce a characterization of $\CaC$-semigroups. This characterization depends on the minimal generating set of the given semigroup. Besides, from its proof, we provide an algorithm for checking if a semigroup is a $\CaC$-semigroup (Algorithm \ref{algoritmo_check_Csemig}). Note that most of the parts of Algorithm \ref{algoritmo_check_Csemig} can be parallelized at least in $q$ stand-alone processes.

\begin{theorem}\label{main_theorem}
A semigroup $S$  minimally generated by $\Lambda_S=\{\mathbf{s}_1,\ldots ,\mathbf{s}_n\}$ is a $\CaC_S$-semigroup if and only if:
\begin{enumerate}
\item $\tau_i \cap (\CaC_S\setminus S)$ is finite for all $i\in [q]$.
\item $\Lambda_S\cap \upsilon_i(\alpha)\neq \emptyset$ for all $\alpha \in \Gamma _i$ and $i\in [q]$.

\end{enumerate}
\end{theorem}

\begin{proof}
Let $S$ be a $\CaC_S$-semigroup. Trivially, $\tau_i \cap (\CaC_S\setminus S)$ is finite for all $i\in [q]$. Assume that $\Lambda_S\cap \upsilon_i(\alpha)= \emptyset$ for some $\alpha \in \Gamma _i$ and some $i\in [q]$. Since $\alpha \in \Gamma _i$, there exists $Q\in \CaA$ such that $\alpha =\Upsilon _i(Q)$. Besides, $Q+\lambda\mathbf{a}_i\in \CaC_S$ and $\Upsilon _i(Q+\lambda\mathbf{a}_i)=\alpha$ for all $\lambda \in \N$. For some $\lambda\in\N$, $Q+\lambda\mathbf{a}_i$ has to be in $S$ ($S$ is $\CaC_S$-semigroup), that is to say, $Q+\lambda\mathbf{a}_i=\sum_{j\in [n]} \mu_j \mathbf{s}_j$ with $\mu_1,\ldots ,\mu_n\in \N$. Therefore, $\alpha=\Upsilon_i (Q+\lambda\mathbf{a}_i)= \sum_{j\in [n]} \mu_j \Upsilon_i(\mathbf{s}_j)$.
By Lemma \ref{lemma_diamante}, for all $j\in [n]$, $\mathbf{s}_j=Q_j+\sum_{k\in [q]} \beta_{jk}\mathbf{a}_k$ for some $Q_j\in \CaA$ and $\beta_{j1},\ldots ,\beta_{jq}\in \N$. So, $\alpha= \sum_{j\in [n]} \mu_j \Upsilon_i(Q_j+\sum_{k\in [q]} \beta_{jk}\mathbf{a}_k)= \sum_{j\in [n]} \mu_j \Upsilon_i(Q_j) +\sum_{j\in [n]} \sum_{k\in [q]} \mu_j  \beta_{jk}\Upsilon_i( \mathbf{a}_k)$. Since $\alpha$ is a minimal element in $ T_i$, $\sum_{j\in [n]} \mu_j +\sum_{j\in [n]} \sum_{k\in [q]\setminus \{i\}} \mu_j  \beta_{jk}=1$. Hence, there exists $\mathbf{s}\in\Lambda_S$ such that $\Upsilon_i(\mathbf{s})=\alpha$ and then $\Lambda_S\cap \upsilon_i(\alpha)\neq \emptyset$.

Conversely, we assume that $\forall i\in [q]$ and $\forall \alpha \in \Gamma _i$, $\tau_i \cap (\CaC_S\setminus S)$ is finite and $\Lambda_S\cap \upsilon_i(\alpha)\neq \emptyset$ (recall that $\mathbf{c}_i=\mathbf{f}_i+\mathbf{a}_i$).
The second condition implies that for $\beta=\Upsilon_i(Q)$ with $Q\in \CaD$, each line $\upsilon_{i}(\beta)$ included an unique non zero minimum (respect 1-norm) point belonging to $S$. Denote by $\{\mathbf{m}_{i1},\ldots , \mathbf{m}_{id_i}\}$ the set obtained from the union of above points for the different elements in $\CaD$ (some of these elements belong to $\Lambda_S$). Note that $\mathbf{m}_{ij}+\mathbf{c}_i+\lambda \mathbf{a}_i\in S$ for all $j\in [d_i]$ and $\lambda\in \N$. Consider $n_i:=\max \{||\mathbf{m}_{i1}+\mathbf{c}_i||_1,\ldots , ||\mathbf{m}_{id_i}+\mathbf{c}_i ||_1\}$, and $\mathbf{x}_i$ the minimum (respect to the 1-norm) element in $\tau_i\cap S$ such that $||\mathbf{x}_i||_1$ is greater than or equal to $n_i$. The set $\CaD_i:=\CaD+\mathbf{x}_i$ satisfies that $\CaD_i \cap S=\CaD_i \cap \CaC_S=\CaD_i$, so $\mathbf{x}_i+\CaC_S\subset S$. We define the finite set $\CaX:= \{\sum_{i\in [q]}\lambda _i\mathbf{x}_i \mid 0\le \lambda_i\le 1 \} $. Since $\mathbf{x}_i+\CaC_S\subset S$ for every $i\in [q]$, $\CaC_S\setminus S\subset \CaX$. Therefore, $S$ is a $\CaC_S$-semigroup.
\end{proof}

\begin{algorithm}[h]\label{algoritmo_check_Csemig}
	\KwIn{The minimal generating set $\Lambda_S$ of a semigroup $S\subset \N^p$.}
	\KwOut{Check if $S$ is a $\CaC_S$-semigroup.}

\Begin{
    $q\leftarrow $ number of extremal rays of $L(S)$\;
    \If{$\tau_i \cap (\CaC_S\setminus S)$ is not finite for some $i\in [q]$}
        {
        \Return $S$ is not a $\CaC_S$-semigroup.
        }

	Compute the set $\{\mathbf{a}_1,\ldots ,\mathbf{a}_q\}$ from $L(S)$\;
	$\CaA\leftarrow \{ \sum_{i\in[q]} \lambda_i\mathbf{a}_i \mid 0\le \lambda_i\le 1\}\cap \N^p$\;
    \ForAll{$i\in [q]$}
        {$\Gamma_i \leftarrow $ the minimal generating set of $ T_i$ obtained from the finite set $\Upsilon _i(\CaA)$\;
        }

    \If{$\Lambda_S\cap \upsilon_i(\alpha)\neq \emptyset$ for all $\alpha \in \Gamma _i$ and $i\in [q]$}
        {
        \Return $S$ is a $\CaC_S$-semigroup.
        }
	\Return $S$ is not a $\CaC_S$-semigroup.
}
\caption{Test if a semigroup $S$ is a $\CaC_S$-semigroup.}
\end{algorithm}

Example \ref{initial_example} illustrates Theorem \ref{main_theorem} and Algorithm \ref{algoritmo_check_Csemig}.

\begin{example}\label{initial_example}
Let $S\subset \N^3$ be the semigroup minimally generated by
$$
\begin{multlined}
\Lambda_S=\{
(2, 0, 0), (4, 2, 4), (0, 1, 0), (3, 0, 0), (6, 3, 6), (3, 1, 1), (4,
1, 1),\\ (3, 1, 2), (1, 1, 0), (3, 2, 3), (1, 2, 1)
 \}.
\end{multlined}
$$
The cone $L(S)$ is $\langle (1,0,0),(2,1,2),(0,1,0) \rangle_{\R_\ge}$ and its supported hyperplanes are $h_1(x,y,z)\equiv 2y-z =0$, $h_2(x,y,z)\equiv x-z =0$ and $h_3(x,y,z)\equiv z=0$.
Recall $\CaC_S=L(S)\cap \N^3$. By $\mathbf{a}_1$, $\mathbf{a}_2$ and $\mathbf{a}_3$ we denote the vectors $(1,0,0)$, $(2,1,2)$ and $(0,1,0)$ respectively, and $\tau_1$, $\tau_2$ and $\tau_3$ are the extremal rays with sets of defining equations $\{h_1(x,y,z)=0,h_3(x,y,z)=0\}$, $\{h_1(x,y,z)=0,h_2(x,y,z)=0\}$ and $\{h_2(x,y,z)=0,h_3(x,y,z)=0\}$, respectively. Hence, $S_1=(\tau_1\setminus \{(1,0,0)\})\cap \N^3$, $S_2=\tau_2\setminus\{(2,1,2)\}\cap \N^3$ and $S_3=\tau_3\cap \N^3$, and the first condition in Theorem \ref{main_theorem} holds.

The set $\CaA$ is equal to
\begin{equation}\label{set_A}
\begin{multlined}
\{(0, 0, 0), (0, 1, 0), (1, 0, 0), (1, 1, 0), (1, 1, 1), (2, 1, 1), (2,1, 2),\\
(2, 2, 2), (3, 1, 2), (3, 2, 2)\},
\end{multlined}
\end{equation}
and
$$\Upsilon_1(\CaA)=\{(0, 0), (0, 2), (1, 1), (2, 0), (2, 2)\},$$
$$\Upsilon_2(\CaA)=\{(0, 0), (0, 1), (1, 0), (1, 1), (2, 0), (2, 1)\},$$
$$\Upsilon_3(\CaA)=\{(0, 0), (0, 1), (0, 2), (1, 0), (1, 1), (1, 2) \}.$$
Therefore, $\Gamma_1=\{(0,2),(1,1),(2,0)\}$ and $\Gamma_2=\Gamma_3=\{(0,1),(1,0)\}$.

Since
$\Upsilon_1(\{(3,1,1),(3,1,2),(1,1,0)\})=\Gamma_1$, $\Upsilon_2(\{(3,1,2),(3,2,3)\})=\Gamma_2$, and $\Upsilon_3(\{(1,1,0),(1,2,1)\})=\Gamma_3$, $S$ satisfies the second condition in Theorem \ref{main_theorem}. Hence, $S$ is a $\CaC_S$-semigroup.

By using our implementation of Algorithm \ref{algoritmo_check_Csemig}, we can confirm that $S$ is a $\CaC_S$-semigroup,

\begin{verbatim}
In [1]: IsCsemigroup([[2,0,0],[4,2,4],[0,1,0],[3,0,0],[6,3,6],
            [3,1,1],[4,1,1],[3,1,2],[1,1,0],[3,2,3],[1,2,1]])
Out[1]: True
\end{verbatim}

\end{example}

To finish this section, it should be pointed out that there exist some special cases of semigroups where Theorem \ref{main_theorem} can be simplified: $\N^p$-semigroups and two-dimensional case.

Note that, if the integer cone $\CaC_S$ is $\N^p$, its supported hyperplanes are $\{x_1=0,\ldots ,x_p=0\}$. Moreover, since its extremal rays are the axes, $\tau_i\equiv \{\lambda \mathbf{e}_i\mid \lambda \in \Q_\ge \}$ is determined by the equations $\cup _{j\in[p]\setminus\{i\}}\{x_{j}=0\}$, and for any canonical generator $\mathbf{e}$ of $\N^{p-1}$, there exists $P$ in $\N^p$ such that $\Upsilon_i(P)=\mathbf{e}$. Furthermore, $\cup _{j\in [p]\setminus \{i\}} \{\Upsilon_i(\mathbf{e}_j)\}$ is the canonical basis of $\N^{p-1}$. Hence, $\Gamma_1=\cdots =\Gamma_p$ is the canonical basis of $\N^{p-1}$. From previous considerations, a characterization of $\N^p$-semigroups equivalent to \cite[Theorem 2.8]{CFU} is obtained from Theorem \ref{main_theorem}.

\begin{corollary}
A semigroup $S$  minimally generated by $\Lambda_S$ is an $\N^p$-semigroup if and only if:
\begin{enumerate}
\item for all $i\in [p]$, the non null entries of the elements in $\tau_i \cap \Lambda_S$ are coprime, or $\mathbf{s}_i=\mathbf{e}_i$.
\item for all $i,j\in [p]$ with $i\neq j$, $\mathbf{e}_i+\lambda_j\mathbf{e}_j\in \Lambda_S$ for some $\lambda_j\in \N$.
\end{enumerate}
\end{corollary}

Focus on two dimensional case, note that the extremal rays and the supported hyperplanes of a cone are equal. Since for each extremal ray the coefficients of its defining linear equation are relatively primes, the linear equations $h_1(x,y)=1$ and $h_2(x,y)=1$ always have non-negative integer solutions. So, any semigroup $S\subset \N^2$ is a $\CaC_S$-semigroup if and only if $\tau_i \cap (\CaC_S\setminus S)$ is finite for $i=1,2$, and both sets $\Lambda_S\cap \{h_1(x,y)=1\}$ and $\Lambda_S\cap \{h_2(x,y)=1\}$ are non empty.

\section{Set of gaps of $\CaC$-semigroups}\label{sec_gaps}

This section gives an algorithm to compute the set of gaps of a $\CaC$-semigroup. This algorithm is obtained from Theorem \ref{main_theorem}. To introduce such an algorithm, let us start by redefining some objects used to prove that theorem.

Given $S$ a $\CaC_S$-semigroup with $q$ extremal rays, for any $i\in [q]$, let $\mathbf{c}_i$ be the conductor of the semigroup $\tau_i\cap S$.
By Corollary \ref{minimal_terminos_independientes}, for any
$\alpha\in \Upsilon_i(\CaD)$ the intersection $\upsilon_i(\alpha)\cap S$ is not empty. Hence, set $\mathbf{m}^{(i)}_\alpha$ the element in $\upsilon_i(\alpha)\cap S$ with minimal 1-norm and $\alpha\in \Upsilon_i(\CaD)\setminus\{0\}$. Note that $\mathbf{m}^{(i)}_\alpha + \mathbf{c}_i+\lambda \mathbf{a}_i\in S$ for all $\lambda\in \N$. Let $n_i:=||\mathbf{c}_i ||_1+ \max \big(\{||\mathbf{m}^{(i)}_\alpha||_1\mid \alpha\in \Upsilon_i(\CaD) \setminus\{0\} \}\big)$,
and $\mathbf{x}_i$ the minimal element in $\tau_i\cap S$ such that $||\mathbf{x}_i||_1$ is greater than or equal to $n_i$. The vector $\mathbf{x}_i$ can be computed as follows: let $Q$ be the non-negative rational solution of the systems of linear equations $\{x_1+\cdots +x_p=n_i, h_{j_1}(x)=0,\ldots, h_{j_{p-1}}(x)=0\}$ (recall that $h_{j_1}(x)=0,\ldots, h_{j_{p-1}}(x)=0$ are the equations defining $\tau_i$), then $\mathbf{x}_i= \Big\lceil \frac{||Q||_1}{||\mathbf{a}_i||_1}\Big\rceil \mathbf{a}_i$.

By the proof of Theorem \ref{main_theorem}, $\CaC_S\setminus S\subset \CaX$, with $\CaX= \{\sum_{i\in [q]}\lambda _i\mathbf{x}_i \mid 0\le \lambda_i\le 1 \}$. Algorithm \ref{computing_gaps_Csemig} shows the process to computed the set of gaps of $S$. Note that several of its steps can be computed in a parallel way.
\begin{algorithm}[h]\label{computing_gaps_Csemig}
	\KwIn{The minimal generating set $\Lambda_S$ of a $\CaC$-semigroup $S\subset \N^p$.}
	\KwOut{Set of gaps of $S$.}

\Begin{
    $\CaH\leftarrow\emptyset$\;
    $q\leftarrow $ number of extremal rays of $L(S)$\;
    \ForAll{$i\in [q]$}
        {
        $\mathbf{c}_i\leftarrow$ conductor of $\tau_i\cap S$\;
        }

    $\CaD\leftarrow \{ \sum_{i\in[q]} \lambda_i\mathbf{s}_i \mid 0\le \lambda_i\le 1\}\cap \N^p$\;

    \ForAll{$i\in [q]$}
        {
        $\Upsilon=\{\alpha_1,\ldots ,\alpha_j\}\leftarrow \Upsilon_i(\CaD)\setminus\{0\}$\;

        \ForAll{$h\in [j]$}
            {
            $\mathbf{m}_h\leftarrow$ the element in $\upsilon_i(\alpha_h)\cap S$ with minimal 1-norm\;
            }

        $n\leftarrow ||\mathbf{c}_i ||_1+ \max \big(\{||\mathbf{m}_1||_1,\ldots ,||\mathbf{m}_j||_1\}\big)$\;

        $\mathbf{x}_i\leftarrow$ minimal element in $\tau_i\cap S$ with $n\le ||\mathbf{x}_i||_1$\;
        }

    $\CaX\leftarrow\{\sum_{i\in [q]}\lambda _i\mathbf{x}_i \mid 0\le \lambda_i\le 1 \} \cap \N^p$\;

    \While{$\CaX\ne \emptyset$}
        {$Q\leftarrow \text{First}( \CaX)$\;

        \If{$Q\notin S$}
            {
            $\CaH\leftarrow \CaH\cup \{Q\}$
            }

        $\CaX\leftarrow \CaX\setminus \{Q\}$\;
        }

	\Return $\CaH$ set of gaps of $S$.
}
\caption{Computing the set of gaps of a $\CaC$-semigroup.}
\end{algorithm}

We illustrate Algorithm \ref{computing_gaps_Csemig} in the following example. Besides, we confirm our handmade computations by using our free software \cite{PROGRAMA}.

\begin{example}\label{ejemploHuecos}
Consider the $\CaC_S$-semigroup $S$ defined in example \ref{initial_example}. So, $\mathbf{s}_1=\mathbf{c}_1= (2,0,0)$, $\mathbf{s}_2=\mathbf{c}_2= (4,2,4)$,  $\mathbf{s}_3= (0,1,0)$ and $\mathbf{c}_3= (0,0,0)$. The set $\CaD$ is
$$
\begin{multlined}
\{
(0, 0, 0), (0, 1, 0), (1, 0, 0), (1, 1, 0), (1, 1, 1), (2, 0, 0), (2,
1, 0), (2, 1, 1),\\ (2, 1, 2), (2, 2, 2), (3, 1, 1), (3, 1, 2), (3, 2,
2), (3, 2, 3), (4, 1, 2), (4, 2, 2),\\ (4, 2, 3), (4, 2, 4), (4, 3, 4),
(5, 2, 3), (5, 2, 4), (5, 3, 4), (6, 2, 4), (6, 3, 4)
\}.
\end{multlined}
$$

For example, for the extremal ray $\tau_1$, $\Upsilon_1(\CaD)$ is the set
$$
\begin{multlined}
\{(0, 0), (0, 2), (0, 4), (1, 1), (1, 3), (2, 0), (2, 2), (2, 4)\},
\end{multlined}
$$
and $\cup_{\alpha\in \Upsilon_1(\CaD) \setminus\{0\}} \{\mathbf{m}^{(1)}_\alpha\}$ is
$$
\begin{multlined}
\{ (0, 1, 0), (3, 1, 1), (3, 1, 2), (3, 2, 2), (3, 2, 3), (4, 2, 4), (4, 3, 4) \}
\end{multlined}$$
For $\tau_2$ and $\tau_3$,
$$
\begin{multlined}
\cup_{\alpha\in \Upsilon_2(\CaD)\setminus\{0\}} \{\mathbf{m}^{(2)}_\alpha\}=
\{(0, 1, 0), (3, 1, 2), (1, 1, 0), (3, 2, 3), (2, 0, 0),\\ (2, 1, 0), (6, 3, 5), (3, 1, 1)\}
\end{multlined}$$
$$\begin{multlined}\cup_{\alpha\in \Upsilon_3(\CaD)\setminus\{0\}} \{\mathbf{m}^{(3)}_\alpha\}=
\{ (1, 1, 0), (1, 2, 1), (2, 0, 0), (2, 3, 1), (2, 4, 2),\\ (3, 1, 1), (3, 1, 2), (3, 2, 3), (4, 2, 2), (4, 3, 3), (4, 2, 4), (5, 3, 4), (6, 2, 4) \}
\end{multlined}$$
Then $n_1=13$, $n_2=24$ and $n_3=12$, and $\mathbf{x}_1= (14,0,0)$, $\mathbf{x}_2= (10,5,10)$ and $\mathbf{x}_3= (0,13,0)$. Therefore, the set of gaps of $S$ is,
$$
\begin{multlined}
\{ (1,0,0), (1,1,1), (2,1,1), (2,1,2), (2,2,1), (2,2,2), (2,3,2),\\ (4,1,2), (4,2,3), (5,2,4), (5,3,5), (8,4,7) \}.
\end{multlined}$$

By using our implementation of Algorithm \ref{computing_gaps_Csemig}, we obtain those gaps,

\begin{verbatim}
In [1]: ComputeGaps([[2,0,0],[4,2,4],[0,1,0],[3,0,0],[6,3,6],
        [3,1,1],[4,1,1],[3,1,2],[1,1,0],[3,2,3],[1,2,1]])
Out[1]: [[1,0,0], [1,1,1], [2,1,1], [2,1,2], [2,2,1], [2,2,2],
        [2,3,2], [4,1,2], [4,2,3], [5,2,4], [5,3,5], [8,4,7]]
\end{verbatim}

\end{example}

\section{Embedding dimension of $\CaC$-semigroups}\label{sec_embedding_dimension}

In \cite{Csemigroup}, it is proved that the embedding dimension of an $\N^p$-semigroup is greater than or equal to $2p$, and this bound holds. Furthermore, a conjecture about a lower bound of embedding dimension of any $\CaC$-semigroup is proposed. In this section, we determine a lower bound of the embedding dimension of a given $\CaC$-semigroup by studying its elements belonging to $\CaA$.

As in previous sections, let $\CaC\subset \N^p$ be a finitely generated cone and $\tau_1,\ldots ,\tau_q$ its extremal rays. For any $i\in [q]$, $\mathbf{a}_i$ is the generator of $\tau_i\cap \N^p$, $\CaA$ is the finite set $\{ \sum_{i\in[q]} \lambda_i\mathbf{a}_i \mid 0\le \lambda_i\le 1\}\cap \N^p$ and $\Gamma_i=\{\alpha^{(i)}_1,\ldots ,\alpha^{(i)}_{m_i} \}$ denotes the minimal generating set of the semigroup $T_i\subset \N^{p-1}$ generated by $\Upsilon_i(\CaA)$. Given a $\CaC$-semigroup $S$, consider $\Lambda':=\{\mathbf{s}_{t_1},\ldots ,\mathbf{s}_{t_k}\}$ the set of minimal elements of $S$ in $\CaA$, and $M_l:=\{i\in [q] \mid \Upsilon_{i}(\mathbf{s}_{t_l})\in\Gamma_i\cup \{0\} \}$ for $l\in [k]$.

The following result provides a lower bound for the embedding dimension of any $\CaC$-semigroup such that $\Lambda'$ is the set of its minimal elements in $\CaA$.

\begin{proposition}\label{proposition_dim}
Given $S\subset \N^p$ a $\CaC$-semigroup,
$$\e(S)\ge \sum _{i\in [q]} (\e(S_i)+\e(T_i))+k-\sum_{i\in [k]}\sharp(M_i).$$
\end{proposition}

\begin{proof}
From Theorem \ref{main_theorem}, for any $i\in [q]$, there exist $\e(S_i)$ minimal generators of $S$ in $\tau_i$. Moreover, for each element $\gamma\in \Gamma_i$, there is at least an element of $\Lambda_S$ in $\upsilon_i(\gamma)$. But, it is possible that one element in $\Lambda_S\cap \CaA$ belongs to two (or more) different lines $\upsilon_i(\gamma)$ and $\upsilon_j(\gamma')$ with $\gamma\in\Gamma_i\cup \{0\} $ and $\gamma'\in\Gamma_j\cup \{0\} $ (in that case, $\upsilon_i(\gamma)\cap \upsilon_j(\gamma')$ is this minimal generator). Since each one of these points in $\Lambda_S\cap \CaA$ can be the only minimal generator of $S$ in those lines, $\sharp(M_l)=n>1$ means that one minimal generator can be the only minimal generator for $n$ different elements in $\cup _{i\in[q]}\Gamma_i\cup \{0\} $. So, counting the minimal amount of elements needed to have almost one minimal generator in each line $\upsilon_i(\gamma)$ for each $\gamma\in\Gamma_i\cup \{0\} $ and $i\in [q]$, we have that the embedding dimension of $S$ is greater than or equal to $\sum _{i\in [q]} (\e(S_i)+\e(T_i))+k-\sum_{i\in [k]}\sharp(M_i).$
\end{proof}

\begin{example}
Consider the $\CaC_S$-semigroup $S$ given in example \ref{initial_example}. In that case, $\Lambda'=\{(3,1,2),(0,1,0),(1,1,0)\}$, $\sharp(M_1)=2$ (i.e. $\Upsilon_i(3,1,2)\in \Gamma_i$ for $i=1,2$), $\sharp(M_2)=2$ ($\Upsilon_1 (0,1,0) \in \Gamma_1$ and $\Upsilon_3 (0,1,0)=(0,0,0)$), and $\sharp(M_3)=2$ ($\Upsilon_1 (1,1,0) \in \Gamma_1$ and $\Upsilon_3 (1,1,0) \in \Gamma_3$). So, $\sum _{i\in [q]} (\e(S_i)+\e(T_i))+k-\sum_{i\in [k]}\sharp(M_i)= 5 + 7 + 3 - 2-2-2= 9$
that is smaller than $\e(S)=11$.
\end{example}

Given any bound, the first interesting question about it is if the bound is reached for some $\CaC$-semigroup. The answer is affirmative, and this fact is formulated as follows.

\begin{lemma}
Let $S_1,\ldots ,S_q$ be the non proper numerical semigroups minimally generated by $\{n_1^{(i)},\ldots ,n_{\e(S_i)}^{(i)}\}$ for each $i\in [q]$, and $\Lambda''\subset \CaC\setminus\{\mathbf{a}_1,\ldots ,\mathbf{a}_q\}$ be a set satisfying
\begin{itemize}
\item  for every $\gamma\in\Gamma_i$ and $i\in [q]$, there exists an unique $\mathbf{d}\in \Lambda''$ such that $\Upsilon_i(\mathbf{d})=\gamma$,
\item if there exist $i,j\in [q]$ and $\mathbf{d},\mathbf{d}'\in \Lambda''$ such that $\Upsilon_i(\mathbf{d})=\Upsilon_j(\mathbf{d}')$, then $\mathbf{d}=\mathbf{d}'$.
\end{itemize}
Then, the embedding dimension of the $\CaC$-semigroup generated by $$\Lambda'' \cup \bigcup_{i\in [q]} \{n_1^{(i)}\mathbf{a}_i,\ldots ,n_{\e(S_i)}^{(i)}\mathbf{a}_i\}$$ is
$$\sum _{i\in [q]} (\e(S_i)+\e(T_i))+k-\sum_{i\in [k]}\sharp(M_i).$$
\end{lemma}

\begin{proof}
By the hypothesis, there are exactly $\sum _{i\in [q]} \e(T_i)+k-\sum_{i\in [k]}\sharp(M_i)$ minimal generators in the $\CaC$-semigroup generated by $$\Lambda'' \cup \bigcup_{i\in [q]} \{n_1^{(i)}\mathbf{a}_i,\ldots ,n_{\e(S_i)}^{(i)}\mathbf{a}_i\}$$ outside its extremal rays, and $\sum_{i\in [q]} \e(S_i)$ belonging to its extremal rays.

\end{proof}

\begin{example}
Let $S\subset \N^3$ be the semigroup minimally generated by
$$
\begin{multlined}
\Lambda_S=\{
(2, 0, 0), (4, 2, 4), (0, 2, 0), (3, 0, 0), (6, 3, 6), (0, 3, 0), (3,
1, 1),\\ (3, 1, 2), (1, 1, 0), (3, 2, 3), (1, 2, 1)
 \}.
\end{multlined}
$$
Note that the cone $\CaC_S$ is the same as the cone in example \ref{initial_example}. So, $\CaA$, $\Gamma_1$, $\Gamma_2$ and $\Gamma_3$ are the sets given in that example. For the semigroup $S$,
$\Upsilon_1(\{(3,1,1),(3,1,2),(1,1,0)\})=\Gamma_1$, $\Upsilon_2(\{(3,1,2),(3,2,3)\})=\Gamma_2$ and $\Upsilon_3(\{(1,1,0),(1,2,1)\})=\Gamma_3$. Since $(1,1,0),(3,1,2)\in \CaA$, $\e(S)= 11= 6+7+2-2-2 = \sum _{i\in [3]} (\e(S_i)+\e(T_i))+2-\sum_{i\in[2]}\sharp(M_i)$.
\end{example}

Fix a cone $\CaC$, studying the different possibilities to select sets of points $K\subset \CaC$ such that $\cup_{i\in[q]}\Gamma_i$ is the union of the minimal generating set of the semigroup given by $\cup_{Q\in K}\Upsilon_i(Q)$ (for $i$ from 1 to $q$), we can state results like the following:

\begin{corollary}
Let $S_1,\ldots ,S_q$ be the non proper numerical semigroups minimally generated by $\{n_1^{(i)},\ldots ,n_{\e(S_i)}^{(i)}\}$ for each $i\in [q]$, and $\Lambda''\subset \CaC\setminus \CaA$ be a set satisfying that for every $\gamma\in\Gamma_i$ and $i\in [q]$, there exists an unique $\mathbf{d}\in \Lambda'' $ such that $\Upsilon_i(\mathbf{d})=\gamma$.
Then, the embedding dimension of the $\CaC$-semigroup generated by $\Lambda'' \cup \bigcup_{i\in [q]} \{n_1^{(i)}\mathbf{a}_i,\ldots ,n_{\e(S_i)}^{(i)}\mathbf{a}_i\}$ is
$\sum_{i\in [q]} (\e(S_i) + \e(T_i))$.
\end{corollary}

Finally, we illustrate the above result with an example.

\begin{example}
Let $S\subset \N^3$ be the semigroup minimally generated by
$$
\begin{multlined}
\Lambda_S=\{(2, 0, 0), (4, 2, 4), (0, 2, 0), (3, 0, 0), (6, 3, 6), (0, 3, 0), (3,
1, 1),\\ (4, 1, 2), (5, 2, 4), (2, 1, 0), (1, 2, 0), (3, 2, 3), (1, 2, 1) \}.
\end{multlined}
$$
Again, the cone $\CaC_S$ is the cone appearing in example \ref{initial_example}. Note that the elements $(2,0,0)$ and $(3,0,0)$ are in $S_1$, $(4,2,4)$ and $(6,3,6)$ belong to $S_2$, and $(0,2,0)$ and $(0,3,0)$ are in $S_3$. Moreover, $\Upsilon_1(\{ (3, 1, 1), (4, 1, 2), (2, 1, 0) \})=\Gamma_1$, $\Upsilon_2(\{(5, 2, 4), (3, 2, 3)\})=\Gamma_2$, $\Upsilon_3(\{(1, 2, 0),(1, 2, 1)\})=\Gamma_3$, and $\Lambda_S\subset \CaC_S\setminus \CaA$. As previous corollary asserts, $\e(S)=13 = 6 + 7 = \sum_{i\in [3]} (\e(S_i) + \e(T_i))$.
\end{example}

\section{On the decomposition of a $\CaC$-semigroup in terms of irreducible $\CaC$-semigroups}\label{sec_irreducible}

We define the set of pseudo-Frobenius of a $\CaC$-semigroup $S$ as
$\textrm{PF}(S)=\{\mathbf{a}\in \CaH(S)\mid \mathbf{a}+(S\setminus\{0\})\subset S\}$, and the set of special gaps of $S$ as
$\textrm{SG}(S)=\{\mathbf{a}\in \textrm{PF}(S)\mid 2\mathbf{a}\in S\}$.
Note that the elements $\mathbf{a}$ of $\textrm{SG}(S)$ are those elements in $\CaC\setminus S$ such that $S\cup \{\mathbf{a}\}$ is again a $\CaC$-semigroup.

A $\CaC$-semigroup is $\CaC$-reducible (simplifying reducible) if it can be expressed as an intersection of two $\CaC$-semigroups containing it properly (see \cite{resolucion_maxima}). Equivalently, $S$ is $\CaC$-irreducible (simplifying irreducible) if and only if $|\textrm{SG}(S)|\le 1$. This definition generalizes the definitions of irreducible numerical semigroup (see \cite{libro_rosales}) and irreducible $\N^p$-semigroup (see \cite{irreducible_GNS}).

Our decomposition method into irreducible is based on adding to a $\CaC$-semigroup elements of $\textrm{SG}(S)$.
If we repeat this operation, we always reach an irreducible $\CaC$-semigroup or the cone $\CaC$. Since the set of gaps $\CaH(S)$ is finite, this process can be performed only a finite number of times. This allows us to state the following algorithm inspired by \cite[Algorithm 4.49]{libro_rosales}.

By definition, the set $\textrm{SG}(S)$ is obtained from $\textrm{PF}(S)$. If $S$ is determined by its minimal generating set, then $\textrm{PF}(S)$ can be computed from the set $\CaH(S)$ obtained with Algorithm \ref{computing_gaps_Csemig}, or using the two different ways given in \cite[Corollary 9 and Example 10]{resolucion_maxima}.

\begin{algorithm}[h]\label{decomposition}
	\KwIn{The minimal generating set $\Lambda_S$ of a $\CaC$-semigroup $S\subset \N^p$.}
	\KwOut{A decomposition of $S$ into irreducible $\CaC$-semigroups.}

\Begin{
    $I\leftarrow \emptyset$\;
    $C\leftarrow \{S\}$\;

    \While{$C\neq \emptyset$}
        {$B\leftarrow \{S'\cup\{\mathbf{a}\} \mid S'\in C,\, \mathbf{a}\in \textrm{SG}(S')\}$\;
        $B\leftarrow B\setminus \{S'\in B\mid \exists \bar{S}\in I\text{ with }  \bar{S}\subset S'\}$\;
        $I\leftarrow I\cup \{S'\in B\mid S' \text{ is irreducible}\}$\;
        $C\leftarrow \{S'\in B\mid S' \text{ reducible}\}$\;
        }

 	\Return $I$.
}
\caption{Computing a decomposition into $\CaC$-semigroups.}
\end{algorithm}

\begin{example}

Consider the $\CaC$-semigroup $S$ given in examples \ref{initial_example} and \ref{ejemploHuecos}. It is minimally generated by
$$
\begin{multlined}
\Lambda_S=\{
(2, 0, 0), (4, 2, 4), (0, 1, 0), (3, 0, 0), (6, 3, 6), (3, 1, 1), (4,
1, 1),\\ (3, 1, 2), (1, 1, 0), (3, 2, 3), (1, 2, 1)
 \},
\end{multlined}
$$
with
$$
\begin{multlined}
\CaH(S)=\{ (1,0,0), (1,1,1), (2,1,1), (2,1,2), (2,2,1), (2,2,2), (2,3,2),\\ (4,1,2), (4,2,3), (5,2,4), (5,3,5), (8,4,7) \}.
\end{multlined}$$
Hence, $\textrm{PF}(S)=\{(2, 2, 1), (2, 3, 2), (4, 1, 2), (8, 4, 7)\}$, and  $\textrm{SG}(S)$ is equal to $\textrm{PF}(S)$.

Applying Algorithm \ref{decomposition} to $S$, we obtain the decomposition into six irreducible $\CaC$-semigroups, $S=S_1\cap\cdots\cap S_{6}$ where

\begin{itemize}
\item $S_1=\langle (3, 0, 0), (2, 0, 0), (1, 1, 0), (0, 1, 0), (4, 1, 1), (3, 1, 1), (3, 1, 2), (4, 1, 2),$ $(1, 2, 1), (2, 2, 1), (2, 2, 2), (3, 2, 3), (4, 2, 4), (6, 3, 6)
 \rangle$;
\item $S_2=\langle (3, 0, 0), (2, 0, 0), (1, 1, 0), (0, 1, 0), (4, 1, 1), (3, 1, 1), (2, 1, 2), (3, 1, 2),$ $(1, 2, 1), (2, 2, 1), (3, 2, 3)
 \rangle$;
\item $S_3=\langle (1, 0, 0), (0, 1, 0), (2, 1, 1), (3, 1, 2), (1, 2, 1), (3, 2, 3), (4, 2, 4), (5, 3, 5),$ $(6, 3, 6)
 \rangle$;
\item $S_4=\langle (3, 0, 0), (2, 0, 0), (1, 1, 0), (0, 1, 0), (2, 1, 1), (1, 1, 1), (3, 1, 2), (4, 1, 2),$ $(3, 2, 3), (4, 2, 4), (6, 3, 6)
 \rangle$;
\item $S_5=\langle (3, 0, 0), (2, 0, 0), (1, 1, 0), (0, 1, 0), (2, 1, 1), (1, 1, 1), (3, 1, 2), (3, 2, 3),$ $(4, 2, 4), (5, 2, 4), (6, 3, 6)
 \rangle$;
\item $S_6=\langle (3, 0, 0), (2, 0, 0), (1, 1, 0), (0, 1, 0), (4, 1, 1), (3, 1, 1), (2, 1, 2), (3, 1, 2),$ $(1, 2, 1), (3, 2, 3), (4, 2, 3) \rangle$;
\end{itemize}

To get these semigroups we have used our implementation in \cite{PROGRAMA} by typing the following
\begin{verbatim}
Csemigroup([[2,0,0],[4,2,4],[0,1,0],[3,0,0],[6,3,6],[3,1,1],
[4,1,1],[3,1,2],[1,1,0],[3,2,3],[1,2,1]]).DecomposeIrreducible()
\end{verbatim}

\end{example}

\subsubsection*{Acknowledgements}
The authors were partially supported by Junta de Andaluc\'{\i}a research group FQM-366.
The first author was supported by the Programa Operativo de Empleo Juvenil 2014-2020, financed by the European Social Fund within the Youth Guarantee initiative.
The second, third and fourth authors were partially supported by the project MTM2017-84890-P (MINECO/FEDER, UE), and the fourth author was partially supported by the project MTM2015-65764-C3-1-P (MINECO/FEDER, UE).

\end{document}